\documentclass[a4paper,12pt]{article}
\usepackage[utf8]{inputenc}

\usepackage{color}
\usepackage{amssymb}
\usepackage{amsmath}
\usepackage{amsfonts}
\usepackage{amsthm}
\usepackage{url}
\usepackage{algorithm}
\usepackage{algorithmic}
\usepackage{hyperref}
\usepackage{longtable}
\usepackage[all,knot,poly]{xy}

\newcommand{\cG}{\mathcal{G}}

\newcommand{\cC}{\mathcal{C}}
\newcommand{\cF}{\mathcal{F}}

\newcommand{\KK}{\mathbb{K}}

\newcommand{\id}{\text{id}}

\newtheorem{theorem}{Theorem}[section]

\newtheorem{proposition}[theorem]{Proposition}

\theoremstyle{definition}

\theoremstyle{remark}

\newenvironment{example}[1][Example \arabic{section}.\arabic{theorem}.]
{\begin{trivlist}\refstepcounter{theorem}
\item[\hskip \labelsep {\bfseries #1}]}{\hfill $\diamondsuit$\end{trivlist}}

\newcommand{\AAA}{\mbox{$\Bbb A$}}

\newcommand{\RR}{\mbox{$\Bbb R$}}

\title{Dynamic Programming in Topological Spaces}
\author{Merve Nur Cakir, Mehwish Saleemi, Karl-Heinz Zimmermann\footnote{Corresponding autor, k.zimmermann(at)tuhh.de}\\ Hamburg University of Technology\\ 20171 Hamburg, Germany}

\begin{document}
\maketitle

\begin{abstract}
Dynamic programming is a mathematical optimization method and a computer programming method as well.
In this paper, the notion of sheaf programming in topological spaces is introduced and it is demonstrated that it
relates very well to the concept of dynamic programming.
\end{abstract}
\medskip

{\bf AMS Subject Classification} 90C39, 54B40, 54D70
\medskip

{\bf Keywords} Dynamic programming, sheaves, sheaf programming, Noetherian topological space

\section{Introduction}
%

Dynamic programming is a mathematical optimization method and a computer programming method as well.
The method was introduced by Richard Bellman~\cite{bellman} in the 1950s and has applications in several fields.
In each case, problem is broken down into simpler sub-problems in a recursive way.
The two major properties of dynamic programming are overlapping sub-problems and optimal substructure.
Sub-problems that need to be solved in a recursion again and again are solved only once and for all and stored for future use.
Optimal substructure refers to solving the sub-problems optimally.

This paper introduces the notion of sheaf computation which amounts to dynamic programming based on sheaves.
The paper is organized as follows.
Section~2 contains the background on presheaves and sheaves.
Section~3 introduces the notion of sheaf computation.
Section~4 provides some known (and unusual) examples.

\section{Sheaves}
In mathematics, sheafs provide a tool for tracking systematically locally defined data associated to the open sets of a topological space~\cite{ga,tenn}.
This kind of data can be restricted to smaller open sets, and the data assigned to an open set corresponds to all collections of compatible data assigned to collections of
smaller open sets covering the given one.
Sheaves are quite abstract objects and their definition is rather subtle.
They come as sheaves of rings or sheaves of modules depending on the type of assigned data.

Let $X$ be a topological space and let $R$ be a commutative ring with unity.
A {\em presheaf\/} $\cF$ of rings on $X$ consists of the following data:
\begin{itemize}
\item For each open subset $U$ of $X$, there is a ring $\cF(U)$ given as the ring of $R$-valued functions on $U$.
\item For each inclusion of open sets $V\subseteq U$ of $X$, there is a ring homomorphism $\rho_{V,U}:\cF(U)\rightarrow\cF(V)$.
\end{itemize}
The elements of $\cF(U)$ are called {\em sections\/} of $\cF$ over $U$.
A section over $X$ is called a {\em global section}. 
The morphisms $\rho_{V,U}$ are called {\em restriction maps}.
We write $\sigma_{|V}$ instead of $\rho_{V,U}(\sigma)$ and think of it as restricting the mapping $\sigma\in\cF(U)$ to the subset $V$.
The restriction maps fulfill the following properties:
\begin{itemize}
\item $\cF(\emptyset) = 0$.
\item For each open set $U$ of $X$, the restriction map $\rho_{U,U} =\id_U$ is the identity.
\item For each inclusion of open sets $W\subseteq V\subseteq U$, we have $\rho_{W,U} = \rho_{W,V}\circ\rho_{V,U}$ 
\end{itemize}
A presheaf can be viewed as a contravariant functor from the category $\cC(X)$, whose objects are the open sets in $X$ and whose morphisms are the inclusion mappings, 
to the category of rings.
%

A presheaf $\cF$ of rings on $X$ is a {\em sheaf\/} if it satisfies the following two conditions:
\begin{itemize}
\item (Uniqueness) For any open subset $U$ of $X$, any open covering $U = \bigcup_{i\in I} U_i$, and any sections $\sigma,\tau\in \cF(U)$,
if $\sigma_{|U_i}=\tau_{|U_i}$ for all $i\in I$, then $\sigma=\tau$ on $U$.
\item (Gluing) For any open subset $U$ of $X$, any open covering $U = \bigcup_{i\in I} U_i$, 
and any family of sections $\sigma_i\in \cF(U_i)$ for $i\in I$ such that
$\sigma_{i|U_i\cap U_j}= \sigma_{j|U_j\cap U_i}$ for all index pairs $(i,j)$, there exists a section $\sigma\in \cF(U)$ such that
$\sigma_{|U_i}= \sigma_i$ for all $i\in I$.
\end{itemize}
These conditions state that sections which are compatible in the sense of the gluing property can be uniquely glued together.


Sheaves are defined on open sets but the underlying topological space $X$ consists of points.
Fix a point $x\in X$ and take pairs $(U,\sigma)$ where $U$ is an open subset of $X$ with $x\in U$ and $\sigma$ is a section over $U$.
Two such pairs $(U,\sigma)$ and $(V,\tau)$ are equivalent if there is an open subset $W$ with $x\in W\subseteq U\cap V$ 
and $\sigma_{|W}=\tau_{|W}$; this defines an equivalence relation.
The set of all such pairs modulo this equivalence is the stalk $\cF_x$ of $\cF$ at $x$, 
which inherits a ring structure from the rings $\cF(U)$.
The elements (equivalence classes) of $\cF_x$ are the germs of $\cF$ at $x$.

Let $\cF$ and $\cG$ be a sheaves on $X$.
A {\em morphism\/} $\phi:\cF\rightarrow\cG$ of sheaves is a collection of homomorphisms $\phi(U):\cF(U)\rightarrow\cG(U)$, where $U\subset X$ is open, such that 
for each inclusion $U\subset V$ of open sets, the following diagram commutes:
$$
\xymatrix{
\cF(V)\ar[dd]^{\rho_{V,U}}\ar[rr]^{\phi(V)}  & & \cG(V)\ar[dd]^{\rho_{V,U}}  \\
                                             & &                  \\
\cF(U)\ar[rr]^{\phi(U)}                      & & \cG(U)  \\
}
$$


\section{Sheaf Computations}

We describe the notion of sheaf computation in topological spaces.

For this, let $X$ be a topological space, $\cF$ be a sheaf of rings on $X$, and $R$ be a commutative ring with unity.

For each open subset $U$ of $X$, let $\cF(U)$ denote the ring of $R$-valued functions on $U$.
The following procedure defines inductively sections over $X$: 
\begin{enumerate}
\item Base step:
Put $U=\emptyset$ and take the zero function $\sigma\in\cF(U)$. 
\item Inductive step:
Consider the collection $(U_i)_{i\in I}$ of open subsets of $X$ for which sections $\sigma_i$, $i\in I$,
have already been defined.
Take the open covering $U'=\bigcup_{i\in I}U_i$ and pick a minimal open subset $U$ of $X$ which contains $U'$
and for which a section over $U$ has not yet been defined.
By glueing and uniqueness, there is a unique section $\sigma'$ over $U'$ such that 
$\sigma'_{|U_i}=\sigma_i$ for all $i\in I$. 
Now extend the section $\sigma'$ to a section $\sigma$ over $U$ such that $\sigma_{|U'}=\sigma'$.
\end{enumerate}
When a global section $\sigma$ over $X$ is inductively reached, 
we speak of a {\em sheaf computation\/} over $X$ and $R$ 
and the global section $\sigma$ is called the {\em result\/} of the computation.

\begin{proposition}
Sheaf computations are well-defined.
\end{proposition}
\begin{proof}
Suppose $U'$ is the open subset of $X$ for which a section $\sigma'$ has already been defined.
Let $U'\subset U$ and $U'\subset V$, where $U$ and $V$ are minimal open sets containing $U'$.
To show that sheaf computations are well-defined, it would be enough to prove that the extended
sections agree on the intersection.
To see this, take the open covering $U\cap V= U'\cup U''$ with $U''\subset U,V$.
Then $U'\subset U'\cup U''\subset U,V$, which contradicts the minimality of $U,V$.
Hence, $U\cap V=U'$ on the sections coincide.
\end{proof}
Intuitively, function-like objects form a presheaf; 
they give rise to a sheaf if the functions exhibit a local behaviour~\cite{ga}.


Sheaf computations are deterministic when a total ordering on the minimal open subsets $U$ of $X$ containing $U'$ is given.

Sheaf computations are finite if the underlying topological space $X$ is Noetherian.
A Noetherian topological space is a topological space $X$ in which the closed subsets fulfill the descending chain condition.
Equivalently, the open subsets of $X$ satisfy the ascending chain condition, since each open subset is the complement of a closed subset. 
Equivalently, each open subset $U$ of $X$ is compact, i.e., each open cover of $U$ has a finite open subcover.
For instance, each affine variety is a Noetherian topological space~\cite{ga}.


\begin{proposition}
If $X$ is a Noetherian topological space, the procedure of sheaf computation provides a global section over $X$.
\end{proposition}

A topology on $X$ indudes a {\em subspace topology\/} on any subset $Y$ of $X$. 
For this, the open subsets of $Y$ are of the form $Y\cap U$ where $U$ is an open subset of $X$.
Note that if $Y$ is open, the open subsets of $Y$ in the subspace topology are exactly the open subsets of $X$ contained in $Y$.
From this perspective, a sheaf computation of a global section $\sigma$ over $X$ follows the paradigm of dynamic programming.

\section{Examples}
In the following examples, the topological spaces are built up from bases.
The base $B$ for a topological space $X$ is a collection of open sets in $X$ such that each open set in $X$ can be written as a union of elements of $B$.
Bases have two important properties:
(1) The base elements cover $X$.
(2) Let $U_1$ and $U_2$ be base elements and let $U$ be their intersection.
Then for each element $x\in U$, there is a base element $U_3$ with $U_3\subseteq U$ and $x\in U_3$.
In particular, if the base $B$ is closed under intersection, then $U$ is also a base element.

\begin{example}
Let $n\geq 0$ be an integer. 
The set $X=[n] = \{i\mid 0\leq i \leq n\}$ forms a topological space in which (except $U_{-1}=\emptyset$ and $X$) 
the subsets $U_i=\{0,\ldots,i\}$ for $0\leq i\leq n$ are open.
For each integer $i\geq 0$, the open subset $U_i$ of $X$ has $U_{i-1}$ as unique maximal open subset.
Thus an already defined section $\sigma_{i-1}$ over $U_{i-1}$ can only be extended to a section $\sigma_{i}$ over $U_{i}$, 
where $\sigma_i(u) = \sigma_{i-1}(u)$ for each $u\in U_{i-1}$.
\end{example}

\begin{example}
Let $m,n\geq 0$ be integers.
Consider the cartesian product set $X=[m]\times [n] = \{(i,j)\mid 0\leq i,j\leq n\}$.
For each pair $(i,j)\in[m]\times[n]$, define the rectangular set
$$U_{i,j} = \{(k,l)\in X \mid 0\leq k\leq i,0\leq l\leq j\}.$$
Take the sets $U_{i,j}$ as basis of the topological space $X$.
Put $U_{i,-1}=\emptyset$ and  $U_{-1,j}=\emptyset$ for all $i\in[m]$ and $j\in[n]$.
Each open subset $U$ of $X$ has the shape of an irregular staircase (Fig.~\ref{f-sc}).
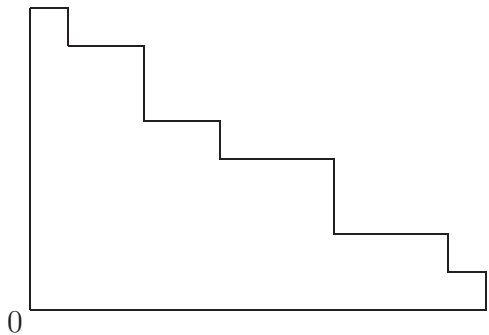
\begin{figure}[hbt]
\begin{center}
\setlength{\unitlength}{1mm}
\begin{picture}(65,60)
\put(5,5){\line(1,0){60}}
\put(5,5){\line(0,1){40}}
\put(5,45){\line(1,0){5}}
\put(10,45){\line(0,-1){5}}
\put(10,40){\line(1,0){10}}
\put(20,40){\line(0,-1){10}}
\put(20,30){\line(1,0){10}}
\put(30,30){\line(0,-1){5}}
\put(30,25){\line(1,0){15}}
\put(45,25){\line(0,-1){10}}
\put(45,15){\line(1,0){15}}
\put(60,15){\line(0,-1){5}}
\put(60,10){\line(1,0){5}}
\put(65,10){\line(0,-1){5}}
\put(2,2){0}
\end{picture}
\end{center}
\caption{Open subset as staircase.}\label{f-sc}
\end{figure}

Given an open subset $U'$ of $X$, a minimal open subset $U$ containing $U'$ has the form $U = U'\cup U_{i,j} = U'\cup \{(i,j)\}$, where $(i,j)\in X$ is the only point not contained in $U'$.
In this way, the section $\sigma'$ on $U'$ can be extended to a section $\sigma$ on $U$ by assigning $(i,j)$ a value $\sigma(i,j)$.
\begin{figure}[hbt]
\begin{center}
\setlength{\unitlength}{1mm}
\begin{picture}(45,45)
\put(10,5){\line(1,0){40}}
\put(10,5){\line(0,1){40}}
\put(10,30){\line(1,0){30}}
\put(10,40){\line(1,0){20}}
\put(30,5){\line(0,1){35}}
\put(40,5){\line(0,1){25}}
\put(7,2){\tiny 0}
\put(2,30){\tiny $j-1$}
\put(7,40){\tiny $j$}
\put(30,2){\tiny $i-1$}
\put(40,2){\tiny $i$}
\end{picture}
\end{center}
\caption{Staircase $U'=U_{i-1,j}\cup U_{i,j-1}$.}\label{f-st}
\end{figure}
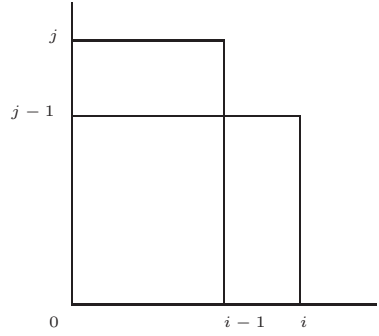

For instance, consider the dynamic programming algorithm of Needleman-Wunsch~\cite{nw} for the alignment of two sequences of length $m$ and $n$.
This algorithm defines a function $\sigma:X\rightarrow\RR$ by setting $\sigma(0,0) = 0$
and for all $(i,j)\in X$,
\begin{eqnarray*}
\sigma(i,j) &=& \min \left\{\begin{array}{l} \sigma(i-1,j)+c(i-1,j), \\ \sigma(i,j-1)+c(i,j-1), \\ \sigma(i-1,j-1)+c(i-1,j-1)\end{array}\right\},
\end{eqnarray*}
where $c:X\rightarrow \RR$ is a function depending on the two sequences to be aligned.
This is a sheaf algorithm defining a global section $\sigma$ over $X$.
Here a section $\sigma'$ over 
$U'=U_{i-1,j}\cup U_{i,j-1}$ 
is extended to a section $\sigma$ over $U_{i,j}$ by defining the value $\sigma(i,j)$ 
in dependence of $\sigma'$ and $c$ (Fig.~\ref{f-st}).
\end{example}


\begin{example}
Let $\AAA^n$ denote the affine $n$-space over a field $\KK$.
Consider the zero set $X\subseteq \AAA^4$ given by the equation $x_1x_4=x_2x_3$~\cite{ga}.
The underlying topology is the Zariski topology and the distinguished open sets form a basis of the Zariski topology
on $X$.

Let $U\subset X$ denote the open subset of all points in $X$ where $x_2\ne 0$ or $x_4\ne 0$.
The quotient $\frac{x_1}{x_2}$ is defined on the set $U_2$ of all points of $X$ where $x_2\ne 0$, 
and the quotient $\frac{x_3}{x_4}$ is defined on the set $U_4$ of all points of $X$ where $x_4\ne 0$.
The sets $U_2$ and $U_4$ are distinguished open sets, but $U=U_2\cup U_4$ is not a distinguished open set.
Both quotients are the same when they are both defined, since then 
$$ \frac{x_1}{x_2} = \frac{x_3}{x_4}.$$

Consider a monadic function $f:\KK\rightarrow\KK$.
Define the sections $\sigma_2:U_2\rightarrow \KK$ and $\sigma_4:U_4\rightarrow \KK$ 
as $\sigma_2(x_1,x_2,x_3,x_4) = f(\frac{x_1}{x_2})$ and $\sigma_4(x_1,x_2,x_3,x_4) = f(\frac{x_3}{x_4})$,
respectively.
It is clear that both functions are equal on the intersection $U_2\cap U_4$.
Since $U=U_2\cup U_4$, we have a section $\sigma':U\rightarrow \KK$ where 
$\sigma'_{|U_2}=\sigma_2$ and $\sigma'_{|U_4}=\sigma_4$.
A global section $\sigma:X\rightarrow\KK$ can be defined say by defining $\sigma(x_1,x_2,x_3,x_4)=1$ for all
points in $X$ which are not in  the union $U_2\cup U_4$ and $\sigma_{|U} = \sigma'$.
\end{example}

\end{document}